\documentclass[11pt,thmsa]{article}

\usepackage{amssymb,latexsym}
\usepackage{amsmath,amscd}
\usepackage[all]{xy}

 \usepackage[latin1]{inputenc}

 \usepackage{mathrsfs}
 \usepackage{dsfont}
 \usepackage{amsmath}
 \usepackage{amsthm}
 \usepackage{amsfonts}
 \usepackage{amssymb}
 \usepackage[pdftex]{graphicx}
 \usepackage{wrapfig}
 \usepackage{layout}
 \usepackage{verbatim}
 \usepackage{alltt}
 \usepackage{xypic}
 \input xy
 \xyoption{all}

 \DeclareMathAlphabet{\mathpzc}{OT1}{pzc}{m}{it}

 \newtheorem{theorem}{Theorem}[section]
 \newtheorem{lemma}[theorem]{Lemma}
 \newtheorem{proposition}[theorem]{Proposition}
 
 \newtheorem{corollary}[theorem]{Corollary}
 \newtheorem{definition}[theorem]{Definition}
  \theoremstyle{definition}
 
 \newtheorem{remark}[theorem]{Remark}

\newtheorem*{conventions}{Conventions}

\newtheorem*{acknowledgements}{Acknowledgements}
\renewenvironment{proof}{\noindent{\it
Proof.}}{\bgroup\hspace{\stretch{1}}$\square$\egroup\medskip\par}

\newcommand{\Rep}{\textrm{Rep}}

\newcommand{\D}{{\bf \Delta}}
\newcommand{\im}{\textrm{im}}
\newcommand{\RRep}{\mathcal{R}\textrm{ep}^{\infty}}
\newcommand{\URRep}{\mathcal{\hat{R}}\textrm{ep}^{\infty}}

\newcommand{\id}{\mathrm{id}}
\newcommand{\End}{\textrm{End}}
\newcommand{\Hom}{\textrm{Hom}}
\newcommand{\RHom}{\underline{\mathrm{Hom}}}

\newcommand{\front}{P}
\newcommand{\back}{Q}
\newcommand{\A}{\mathsf{A}}

\newcommand{\C}{\mathcal{C}}

\begin{document}

\vspace{15cm}
 \title{Reidemeister torsion for flat superconnections}
\author{Camilo Arias Abad\footnote{Institut f\"ur Mathematik, Universit\"at Z\"urich,
camilo.arias.abad@math.uzh.ch. Partially supported
by SNF Grant 20-113439 and the Forschungskredit of the Universit\"at Z\"urich. } \hspace{0cm} and
Florian Sch\"atz\footnote{Center for Mathematical Analysis, Geometry and Dynamical Systems, IST Lisbon, fschaetz@math.ist.utl.pt. 
Partially supported by the FCT through program POCI 2010/FEDER, by post-doctoral grant
SFRH/BPD/69197/2010 and by project PTDC/MAT/098936/2008.}}

\maketitle
\begin{abstract} 
We use higher parallel transport -- more precisely, the integration $\mathsf{A}_\infty$-functor constructed in \cite{BS,AS2} -- to define Reidemeister torsion for flat superconnections. 
We hope that the combinatorial
Reidemeister torsion coincides with the analytic torsion defined by Mathai and Wu \cite{MathaiWu2}, thus permitting
for a generalization of the Cheeger-M\"uller Theorem.
\end{abstract}

\tableofcontents

\section{Introduction}\label{s:intro}

The main goal of this paper is to explain how higher parallel transport can be used to define Reidemeister torsion of 
flat superconnections. The classical Reidemeister torsion \cite{De-Rham, Franz, Reidemeister} is an invariant
of a flat vector bundle $E$ on a closed odd dimensional manifold $M$. It was first introduced by Reidemeister in order to distinguish lens spaces which are homotopy equivalent but
not homeomorphic. This invariant -- which is a norm $\tau_R$ on the determinant line of the cohomology $H(M,E)$ -- is defined by choosing a triangulation of $M$.
The corresponding cellular complex is finite dimensional and computes $H(M,E)$.
The norm $\tau_R$ is then constructed with the help of the cellular basis and the fact that
the determinant line of a finite dimensional complex is naturally isomorphic to the determinant
line of its cohomology.

We prove that given a flat superconnection on an closed, orientable, odd dimensional manifold $M$,  the $\A_\infty$-version of parallel transport constructed in \cite{BS,AS2}, together with the choice of a triangulation, produce  a finite dimensional complex computing the cohomology $H(M,E)$. 
Using this complex and Poincare duality, one can induce a norm  -- the Reidemeister torsion $\tau_R$ -- on the determinant line
$\det H(M,E)$, thus extending the construction
\begin{align*}
(E,\nabla) \mapsto \tau_R
\end{align*}
to the $\mathbb{Z}_2$-graded setting.

The notion of higher parallel transport that we use can be formulated as follows. Ordinary parallel transport for flat vector bundles yields an integration functor:
\[\int: \mathsf{Flat}(M)\rightarrow \Rep(\Pi_1(M)),\]
from the category of flat vector bundles on $M$ to the category of representations of the fundamental groupoid $\Pi_1(M)$ of $M$.
Based on the work of Gugenheim \cite{Gugenheim} and Igusa \cite{I}, Block-Smith \cite{BS} and Arias-Sch\"atz \cite{AS2} showed that this integration functor extends to an $\A_\infty$-functor which makes the following diagram commute:

\begin{align*}
\xymatrix{
\mathsf{Flat}(M) \ar[r]^\int \ar[d]&\Rep(\Pi_1(M))   \ar[d] \\
\URRep(TM)\ar[r]^\int& \URRep(\pi_{\infty}(M)).
}
\end{align*}

Here, $\URRep(TM)$ denotes the $dg$-category of flat superconnections on $M$, $\URRep(\pi_{\infty}(M))$ denotes the category of representations up to homotopy of the simplicial set $\pi_{\infty}(M)$ of smooth singular chains on $M$; the vertical arrows are natural
inclusions and the horizontal arrow at the bottom is the $\A_\infty$-functor mentioned above. This construction gives higher holonomies associated not only to one dimensional simplices, but to simplices of all dimensions. It provides a combinatorial (singular) way to compute
the cohomology of a flat superconnection on $E$. 

By choosing a smooth triangulation $(K,\phi)$ of $M$, one obtains a subsimplicial set $M^K \hookrightarrow  \pi_{\infty} (M)$ of the infinity groupoid of $M$. By restricting the representation $\int(E)$ to $M^K$, one obtains a finite dimensional
complex that computes the cohomology of $H(M,E)$. This complex is used to define the Reidemeister torsion $\tau_R$ of the flat superconnection on $E$.

We should mention that, in a certain sense, the Reidemeister torsion of a flat superconnection on $E$ is not a new invariant:
any flat superconnection $D$ on $E$ induces a flat connection on the cohomology bundle $H_\partial(E)$. Moreover, filtering the complex
$\Omega(M,E)$ in a natural way yields a spectral sequence $\mathcal{E}_r$ converging to $H(M,E)$ and whose second page is
\[\mathcal{E}^{p,\overline{q}}_2\cong H(M,H^{\overline{q}}_\partial(E)).\]
This induces a canonical isomorphism at the level of determinant lines

\[\det H(M,E)\cong \det H(M,H_\partial(E)).\]

This turns out to be an isomorphism of metric vector spaces. That is, the torsion of the flat superconnection $D$ is mapped to the usual Reidemeister torsion of the flat bundle $H_\partial(E)$
under this isomorphism.

Mathai and Wu \cite{MathaiWu2} have studied the analytic torsion of flat superconnections. It is natural to hope that a version of the Cheeger-M\"uller theorem holds in this context, namely, that the analytic torsion defined by Mathai-Wu coincides with the Reidemeister torsion
defined in the present paper. Since, as mentioned before, the Reidemeister torsion of flat superconnection reduces 
in a certain sense to the ordinary Reidemeister torsion, this question is equivalent to asking whether the same 
reduction holds true for the analytic torsion of Mathai-Wu. Some results in that direction can be found in \cite{MathaiWu},
see in particular Proposition 5.1. there.
\\

The paper is organized as follows. In Section \S \ref{preliminaries} we review the definitions of flat superconnections,
representations up to homotopy of simplicial sets, and the basic properties of the integration $\A_\infty$-functor constructed in \cite{BS,AS2}.
We also show that, by choosing a triangulation of the manifold $M$, one can find a finite dimensional complex computing the cohomology
of a flat superconnection $E$ on $M$. In Section \S \ref{torsion} we explain how to use the higher version of parallel transport to define
the Reidemeister torsion of a flat superconnection on a closed, orientable, odd dimensional manifold $M$. The main result of the paper is Theorem \ref{theorem:invariance}, which states the independence
of Reidemeister torsion of all auxiliary choices. We prove -- Corollary \ref{equal} -- that under the natural identification of determinant lines, the Reidemeister torsion of a flat superconnection on
$E$ coincides with the ordinary Reidemeister torsion of the flat connections induced on the cohomology bundle $H_\partial(E)$. It is also proven that the Reidemeister torsion is invariant under quasi-isomorphisms of flat superconnections, see Proposition \ref{quasi-iso}.
In the Appendix we collect some general facts regarding homological algebra of determinant lines and
prove some duality results for spectral sequences.

\begin{acknowledgements}
We would like to thank Maxim Braverman, Alberto Cattaneo, Calin Lazaroiu, Pavel Mn\"{e}v and James Stasheff for useful conversations.
We also thank the University of Zurich and the Instituto Superior Tecnico in Lisbon for their hospitality.
\end{acknowledgements}

\begin{conventions}
All complexes we consider are complexes of vector spaces.
Moreover, all isomorphisms between linear objects (complexes, graded vector spaces, vector spaces) are considered
up to sign.
\end{conventions}

\section{Preliminaries}\label{preliminaries}

We discuss results regarding higher notions of parallel transport for flat superconnections. The 
constructions we describe here are based on \cite{AS2,BS,I}. While these papers are written in the setting 
of $\mathbb{Z}$-graded superconnections, we will be interested in the $\mathbb{Z}_2$-graded case. Most of the results
apply in this setting without major modifications. We will indicate the changes that are necessary in the $\mathbb{Z}_2$-graded case.

\subsection{Flat superconnections }\label{section:flat_superconnections}

Let $E$ be a $\mathbb{Z}_2$-graded vector bundle over $M$, i.e. $E = E^{\overline{0}}\oplus E^{\overline{1}}$.
The space $\Omega(M,E)=\Gamma(\wedge T^*M \otimes E)$ of differential forms with values in $E$ is a $\mathbb{Z}_2$-graded vector space with
components
\begin{align*}
\Omega(M,E)^{\overline{0}} := \bigoplus_{\overline{k}+\overline{l}=\overline{0}}\Omega^{k}(M,E^{\overline{l}}) \qquad \textrm{and} \qquad
\Omega(M,E)^{\overline{1}} := \bigoplus_{\overline{k}+\overline{l}=\overline{1}}\Omega^{k}(M,E^{\overline{l}}).
\end{align*}
The vector space $\Omega(M,E)$ is a $\mathbb{Z}_2$-graded module over the algebra $\Omega(M)$.

\begin{definition}
A superconnection on $E$ is a linear operator
\begin{align*}
 D: \Omega(M,E) \to \Omega(M,E)
\end{align*}
of odd degree which satisfies the Leibniz rule
\begin{align*}
D(\alpha \wedge \omega) = d\alpha \wedge \omega + (-1)^{|\alpha|} \alpha \wedge D(\omega)
\end{align*}
for all homogeneous $\alpha \in \Omega(M)$ and $\omega \in \Omega(M,E)$.
A superconnection $D$ is flat if $D^2=0$.
\end{definition}

\begin{remark}
Any superconnection $D$ on $E$ corresponds to a family of operators
\begin{align*}
 (\partial, \nabla, \omega_2, \omega_3, \dots),
\end{align*}
where $\partial$ is a fiberwise linear operator on $E$, $\nabla=(\nabla^{\overline{0}},\nabla^{\overline{1}})$
is a pair of connections on $E^{\overline{0}}$ and $E^{\overline{1}}$, and $\omega_k$
are differential forms of degree $k$ with values in $\End^{\overline{1-k}}(E)$, where $\End(E)$ denotes the $\mathbb{Z}_2$-graded algebra of
endomorphisms of $E$.

The flatness condition on $D$ translates into a family of quadratic relations of the form

\[ \partial \circ \partial = 0, \quad
 [\partial, \nabla] = 0,\quad [\partial, \omega_2] + R_{\nabla} =0, \qquad \textrm{and} \]
\[ [\partial, \omega_k] + [d_{\nabla},\omega_{k-1}] + \sum_{r+s=k, r,s\ge 2} \omega_{r} \wedge \omega_{s} = 0, \quad \textrm{for } k\ge 3. \]

\end{remark}

\begin{definition}
Let $E$ be a $\mathbb{Z}_2$-graded vector bundle, equipped with a flat superconnection $D$.
The cohomology $H(M,E)$ of $M$ with values in $E$ is the cohomology of the complex $(\Omega(M,E),D)$.
\end{definition}

\begin{conventions}
In the following, {\it $dg$-category} will always refer to a {\it $\mathbb{Z}_2$-graded $dg$-category}.
\end{conventions}

\begin{remark}
The flat superconnections on $M$ can be organized into a $dg$-category.
A morphism of parity $\overline{p}$ between two flat superconnections $E$ and $E'$ on $M$  is a parity $\overline{p}$ morphism of
$\Omega(M)$-modules
\[\phi: \Omega(M,E)\rightarrow \Omega(M,E').\]
Observe that we do not require $\phi$ to be a chain map.
The space of morphisms 
\[\underline{\Hom}(E,E')= \underline{\Hom}^{\overline{0}}(E,E') \oplus \underline{\Hom}^{\overline{1}}(E,E'),\]
is a $\mathbb{Z}_2$-graded complex with differential
\[\Delta(\phi):=D' \circ \phi -(-1)^{|\phi|}\phi \circ D. \]
We will denote the resulting $dg$-category by $\URRep(T(M))$.
\end{remark}
\begin{remark}
Since a flat superconnection is an elliptic differential operator, the cohomology $H(M,E)$ is finite dimensional for $M$ a closed manifold.
\end{remark}

\begin{definition}
Let $E$ be a $\mathbb{Z}_2$-graded vector bundle over $M$, equipped with a flat superconnection $D$.
The dual of $D$ is the flat superconnection $D^*$ on $E^*$ determined by the condition:
\begin{align*}
 d<\alpha,\omega> = <D^*\alpha, \omega> + (-1)^{|\alpha|} <\alpha, D\omega>
\end{align*}
for all $\alpha \in \Omega(M,E^*)$ and $\omega \in \Omega(M,E)$.
\end{definition}

\begin{remark}\label{remark:PD_1}
Under the assumption that $M$ is closed and orientable, Poincar\'e duality holds true for flat superconnetions.
That is, the natural pairing
\begin{align*}
\Omega(M,E^*) \otimes \Omega(M,E) \to \mathbb{R}, \qquad (\alpha,\omega) \mapsto \int_{M} <\alpha,\omega>
\end{align*}
induces a perfect pairing
\begin{align*}
H(M,E^*) \otimes H(M,E) \to \mathbb{R},
\end{align*}
as can be verified using Hodge-theory. This can be seen as a special instance of the duality
established in \cite{Block}.

The parity of the corresponding isomorphism
\begin{align*}
\mathcal{D}: H(M,E^*) \to (H(M,E))^*
\end{align*}
coincides with the dimension of $M$. In particular, if $M$ is odd dimensional, we obtain
a parity-preserving isomorphism of $\mathbb{Z}_2$-graded vector spaces
\begin{align*}
\mathcal{D}: H(M,E^*) \to \Pi H(M,E)^*.
\end{align*}
\end{remark}

\subsection{$\mathbb{Z}_2$-graded representation up to homotopy of simplicial sets}

Let $X_{\bullet}$ be a simplicial set with face and degeneracy maps denoted by
\[d_i: X_k\rightarrow X_{k-1} \quad \text{and} \quad s_i: X_k \rightarrow X_{k+1},\] 
respectively.  We will use the notation
\begin{eqnarray*}
\front_i&:=&(d_0)^{k-i}:X_k \rightarrow X_i,\\
\back_i&:=&d_{i+1} \circ \dots \circ d_{k} :X_k \rightarrow X_i,
\end{eqnarray*}
for the maps that send a simplex to its $i$-th back and front face.
The $i$-th vertex of a simplex $\sigma \in X_k$ 
will be denoted $v_i(\sigma)$, or simply $v_i$, when no confusion can arise.
In terms of the above operations, one can write
\[v_i=(\front_0\circ \back_{i})(\sigma). \]

Suppose that $E$ is a $\mathbb{Z}_2$-graded vector bundle over $X_0$, i.e. that there is a $\mathbb{Z}_2$-graded vector space
$E_x$ for each $x \in X_0$.
A cochain $F$ of degree $k$ on $X_{\bullet}$ with values in $E$ is a map:
\begin{align*}
F: X_k \to E,
\end{align*}
such that $F_k(\sigma)\in E_{v_0(\sigma)}$. We denote by $\hat{C}^k(X, E)$ the vector space of normalized cochains, i.e. those cochains which vanish on degenerate simplices. The spaces of $E$-valued cochains is a $\mathbb{Z}_2$-graded vector space:
\[\hat{C}(X,E):=\hat{C}(X,E)^{\overline{0}} \oplus \hat{C}(X,E)^{\overline{1}}, \]
where
\[  \hat{C}(X,E)^{\overline{0}} :=\prod_{\overline{k}+\overline{i}=\overline{0}} \hat{C}^k(X,E^{\overline{i}}) \quad \text{ and } \quad
 \hat{C}(X,E)^{\overline{1}} :=\prod_{\overline{k}+\overline{i}=\overline{1}} \hat{C}^k(X,E^{\overline{i}}). \]

In case the vector bundle is the trivial line bundle $\mathbb{R}$ we will write $\hat{C}(X)$ instead of $\hat{C}(X,\mathbb{R})$.
The space $\hat{C}(X)$ is naturally a $\mathbb{Z}_2$-graded $dg$-algebra with the cup product and the usual simplicial differential defined by
\[\delta(\eta)(\sigma):=\sum_{i=0}^k(-1)^id^*_i(\eta)(\sigma),\]
for $\eta \in \hat{C}^{k-1}(X)$. Given any $\mathbb{Z}_2$-graded vector bundle $E$ over $X_0$, the cup product gives the space $\hat{C}(X,E)$
the structure of a right graded module over the algebra $\hat{C}(X)$.

\begin{definition}
A unital $\mathbb{Z}_2$-graded representation up to homotopy of $X_{\bullet}$ consists of the following data:
\begin{enumerate}
\item A finite rank $\mathbb{Z}_2$-graded vector bundle $E$ over  $X_0$.
\item A linear map of odd parity $D: \hat{C}(X,E)\rightarrow \hat{C}(X,E)$  which is a derivation with respect to the $\hat{C}(X)$-module structure and squares to zero.
\end{enumerate}
The cohomology of $X_\bullet$ with values in $E$, denoted $H(X,E)$, is the cohomology of the complex $(\hat{C}(X,E),D)$.
\end{definition}

\begin{conventions}
In the following, {\it representations up to homotopy} will always refer to {unital \it $\mathbb{Z}_2$-graded representations up to homotopy}.
\end{conventions}

\begin{remark}
The representations up to homotopy of $X_{\bullet}$ form a  $dg$-category. 
Let $E, E'$ be two representations up to homotopy of $X_{\bullet}$.
A parity $\overline{p}$ morphism $\phi \in \RHom^{\overline{p}}(E,E')$ is a parity $\overline{p}$ map of $\hat{C}(X)$-modules
$\phi: \hat{C}(X,E)\rightarrow \hat{C}(X,E').$
The space of morphisms is naturally a $\mathbb{Z}_2$-graded vector space
\[\RHom(E,E')=\RHom^{\overline{0}}(E,E')\oplus \RHom^{\overline{1}}(E,E')\]

with differential
\begin{eqnarray*}
\D:\RHom^{\overline{p}}(E,E')&\rightarrow& \RHom^{\overline{p+1}}(E,E')\\
\phi &\mapsto & D' \circ \phi - (-1)^{|\phi|} \phi \circ D.
\end{eqnarray*}
 
We denote the resulting $dg$-category by $\URRep(X_\bullet )$. 
\end{remark}

\begin{remark}
The category  $\URRep(X_\bullet )$ is functorial with respect to maps of simplicial sets. Namely, if $f: X_\bullet \rightarrow Y_\bullet$ is a morphism of simplicial sets then there is a pull-back $dg$-functor:
\begin{eqnarray*}
f^*:\URRep(Y_\bullet ) & \rightarrow & \URRep(X_\bullet )\\
E &\mapsto &f^*(E)
\end{eqnarray*}
In particular, this $dg$-functor induces a map in cohomology:
\[f^*:H(Y,E)\rightarrow H(X,f^*E) .  \]
\end{remark}

\subsection{Integration}

We will now discuss the notion of parallel transport for flat superconnections.
It generalizes the fact that a flat connection
on a vector bundle corresponds to a representation of the fundamental groupoid $\Pi_1(M)$ of $M$.
We denote the simplicial set of smooth simplices in $M$ by $ \pi_{\infty}(M)$, that is:
\[\pi_{\infty}(M)_k:=\mathsf{Maps}(\Delta_k,M).\]
One of the central results of \cite{AS2,BS}  is the following:

\begin{theorem}\label{theorem:integration_everything}
There is an $\mathsf{A}_{\infty}$-functor 
\begin{align*}
\int: \RRep(TM) \to \URRep(\pi_{\infty}(M))
\end{align*}
between the $dg$-category of flat superconnections on $M$ and the $dg$-category of representations up to homotopy of $\pi_{\infty}(M)$.
Moreover, the map induced in cohomology is an isomorphism.
\end{theorem}

\begin{remark}
\hspace{0cm}
\begin{itemize}
\item[(1)]
Since the constructions of \cite{AS2,BS} deal with the $\mathbb{Z}$-graded case, some comments are in order.
 The integration $\mathsf{A}_\infty$-functor is constructed using the $\mathsf{A}_\infty$ version of de Rham's 
theorem which was constructed by Gugenheim \cite{Gugenheim} using Chen's iterated integrals \cite{C}. The general 
structure of the construction is as follows. Suppose that $E$ is a $\mathbb{Z}_2$-graded {\em trivial} vector bundle over
$E$. Then, a flat superconnection on $E$ corresponds to a Maurer-Cartan element in $\Omega(M,\End(E))$ while a representation 
up to homotopy of $\pi_{\infty}(M)$ on $E$ corresponds to a Maurer-Cartan element in $\hat{C}(\pi_{\infty}(M),\End(M))$. 
By tensoring Gugenheim's $\mathsf{A}_\infty$-morphism with $\End(E)$ one obtains an $\mathsf{A}_\infty$-morphism between these two algebras. 
Therefore, there is  a way to produce Maurer-Cartan elements in $\hat{C}(\pi_{\infty}(M),\End(M))$ out of Maurer-Cartan elements in $\Omega(M,\End(E))$. 
This procedure works in the $\mathbb{Z}_2$-graded case as well as for $\mathbb{Z}$-grading.
Gauge-invariance of Gugenheim's $\mathsf{A}_\infty$-morphism allows to extend this construction to non-trivial bundles, see \cite{AS2} for
the details. 

\item[(2)]
Given a flat superconnection on $M$, one can define the cohomology associated to it in terms of the $dg$-category $ \RRep(TM)$
as follows:
\begin{equation*}
H(TM,E) := H(\RHom(\mathbb{R},E)).
\end{equation*}
Similarly, given a representation up to homotopy  $E$ of $\pi_{\infty}(M)$, one can define the cohomology associated to it as
\begin{equation*}
H(\pi_{\infty}(M),E) := H(\RHom(\mathbb{R},E)).
\end{equation*}
Since the $\A_\infty$ functor
$\int$ sends the trivial representation to the trivial representation it induces a map

\begin{align*}
\int: H(TM,E)\to H(\pi_{\infty}(M),\int(E)).
\end{align*}
The usual de Rham theorem together with a spectral sequence argument imply that this map is an isomorphism.
\end{itemize}
\end{remark}

\subsection{Simplicial cochains}

So far we have seen that the cohomology of a flat superconnection can be computed using singular cohomology.
We will now explain the corresponding cellular complex.
Suppose that $(K,\phi)$ is a smooth triangulation of the closed manifold $M$, i.e. $\phi: |K| \to M$ is a homeomorphism
between the geometric realization of the finite simplicial complex $K$ and $M$,
such that the restriction to every closed simplex of $K$ is smooth.
The existence of such triangulations was proved by Whitehead, see \cite{Whitehead}.
The choice of a total order in the vertices of $K$  gives rise to a subsimplicial set
$\iota:M^K \hookrightarrow \pi_{\infty}M$.

\begin{lemma}\label{singular-simplicial}
Let $M$ be a closed manifold, $(K,\phi)$ a smooth triangulation of
$M$ together with a total ordering on its vertices. Then the functor
\[\iota^*:  \URRep(\pi_{\infty}(M)) \rightarrow \URRep(M^K),\]
induces an isomorphism in cohomology
\[\iota^*: H(\pi_{\infty}(M),E)\rightarrow H(M^K, \iota^*(E))\]
for any representation up to homotopy $E \in \URRep(\pi_{\infty}(M))$.
\end{lemma}

\begin{proof}
There are natural decreasing filtrations 
\[F_p(\hat{C}(\pi_{\infty}(M),E)):=\prod_{k\geq p} \hat{C}^k(X,E),\]
and 
\[\tilde{F}_p(\hat{C}(M^K,\iota^* E)):=\prod_{ k\geq p} \hat{C}^k(X,\iota^*E),\]
which induce spectral sequences $\mathcal{E}^{p,\overline{q}}_r$ and $\tilde{\mathcal{E}}^{p,\overline{q}}_r$. The map $\iota^*$ respects filtrations
and therefore induces a map of spectral sequences

\[i^*: \mathcal{E}^{p,\overline{q}}_r\rightarrow \tilde{\mathcal{E}}^{p,\overline{q}}_r.\]
It suffices to prove that the map of spectral sequences is an isomorphism for $r=2$.
Since $E$ is a representation up to homotopy of $\pi_{\infty}(M)$, the cohomology bundle $H_{\partial}(E)$ 
is a local system over $M$ and moreover

\[  \mathcal{E}^{p,\overline{q}}_2\cong H^p(M, H_{\partial}^{\overline{q}}(E)) . \]
The result then reduces to the usual computation of the cohomology of a local system in terms of a triangulation. 
\end{proof}

By putting together Lemma \ref{singular-simplicial} and Theorem \ref{theorem:integration_everything} we obtain:

\begin{corollary}
Let $M$ be a closed manifold, $E$ a flat superconnection on $M$ and $(K,\phi)$ a smooth triangulation of $M$ together which a
total ordering of its vertices. Then there exists a canonical isomorphism
\[H(M,E)\cong H(M^K, \iota^* \int (E)).\]
\end{corollary}

\begin{remark}
To simplify the notation, when no confusion arises we will denote the complex $\hat{C}(M^K,\iota^*\int(E))$ by $C(M^K,E)$.
Observe that $C(M^K, E)$ is finite-dimensional. In fact, there is a natural isomorphism
\begin{align*}
C(M^K,E) \cong \bigoplus_{\Delta \in K} \Pi^{\dim \Delta} E_{v_0(\Delta)},
\end{align*}
where $\Pi$ is the endofunctor on the category of $\mathbb{Z}_2$-graded vector spaces that reverses the parity.
\end{remark}

\section{Reidemeister Torsion}\label{torsion}
\subsection{Definition of combinatorial torsion}

Let $M$ be a closed, orientable, odd dimensional manifold and $E$ a $\mathbb{Z}_2$-graded vector bundle over $M$, equipped with a flat superconnection $D$.
We will define a norm $\tau_{R}$ -- the {\em Reidemeister torsion} -- on the determinant line $\det H(M,E)$ of the cohomology of $M$ with values in $E$.
We will follow Farber's approach \cite{Faber} to Reidemeister torsion. The construction given below extends the usual Reidemeister torsion of flat vector bundles.

In order to define $\tau_{R}$, we choose a smooth triangulation $(K,\phi)$ of $M$, together with a total ordering of the vertex set of $K$.
As we have seen before, this choice gives a subsimplicial set $M^K \hookrightarrow \pi_{\infty}M$.

\begin{definition}
Let $(K,\phi)$ be a smooth triangulation of $M$, together with a total ordering of its vertex set
and $E$ a $\mathbb{Z}_2$-graded vector bundle over $M$, equipped with a flat superconnection $D$.
The complex $\C_K(M,E)$, associated to these data is 
\[\C_K(M,E):=C(M^K, E)\oplus C(M^K, E^*).\]
\end{definition}

\begin{remark}
Using the standard identities for determinant lines -- see Appendix \ref{appendix1}-- we obtain
\begin{eqnarray*}
\det\mathcal{C}_K(M,E) &\cong& \det C(M^K,E) \otimes \det C(M^K,E^*) \\
&\cong& \det H(M,E) \otimes \det H(M,E^*)\\
&\cong& \det H(M,E) \otimes \det H(M,E).
\end{eqnarray*}
Here we used the isomorphism $\mathcal{D}: H(M,E^*) \to \Pi H(M,E)^*$ corresponding to Poincar\'e-duality, see Remark \ref{remark:PD_1}, in the 
transition from the second to the third line.

We conclude that every norm on $\det \mathcal{C}_K(M,E)$ yields a norm on $(\det H(M,E))^{\otimes 2}$
and then, via the diagonal mapping
\begin{align*}
\det H(M,E) \to (\det H(M,E))^{\otimes 2}, \qquad x \mapsto x\otimes x
\end{align*}
a norm on $\det H(M,E)$.
\end{remark}

\begin{lemma}\label{lemma:fiber_cohomology}
Let $E$ be a $\mathbb{Z}_2$-graded vector bundle over $M$ equipped with a flat superconnection $D$.
Denote the corresponding fiberwise differential on $E$ by $\partial$ and the corresponding connection on $E$ by $\nabla$.
The cohomology bundle $H_{\partial}(E)$ inherits a flat connection $[\nabla]$.

The canonical isomorphism of vector bundles
\begin{align*}
\det E \cong \det H_{\partial}(E)
\end{align*}
maps $\det\nabla$ to $\det[\nabla]$. In particular, the connection $\det\nabla$ is flat.
\end{lemma}

\begin{proof}
Recall that the isomorphism $\det E \cong \det H_{\partial}(E)$ is a consequence of the following two exact sequences
\begin{align*}
 \xymatrix{
0 \ar[r] & B(E) \ar[r] & Z(E) \ar[r] & H_{\partial}(E) \ar[r] & 0 \\
0 \ar[r]  & Z(E) \ar[r] & E \ar[r]^{\partial} & \Pi B(E) \ar[r] & 0,
}
\end{align*}
where $Z(E)$ and $B(E)$ are the vector bundles of fiberwise closed and exact elements of $E$, respectively. 
By compatibility with the fiberwise differential, the connection $\nabla$ on $E$ induces connections on $Z(E)$, $B(E)$ and $H_{\partial}(E)$.
It is easy to check that both short exact sequences become short exact sequences of vector bundles with connections.

Hence we obtain
\begin{align*}
\det E \cong \det H_{\partial}(E) \otimes \det B \otimes \det \Pi B
\end{align*}
as line bundles with connections.
Since
\begin{align*}
\det B \otimes \Pi B \cong \End(\det B) \cong \underline{\mathbb{R}},
\end{align*}
where $\underline{\mathbb{R}}$ denotes the trivial line bundle with its trivial connection, are isomorphisms of line bundles with connections,
the claim follows.
\end{proof}
\begin{definition}
Let $E$ be a $\mathbb{Z}_2$-graded vector bundle over $M$, equipped with a flat superconnection $D$.
The determinant bundle associated to $E$ is the line bundle $\det E$, equipped with the flat connection $det \nabla$
induced from the connection $\nabla$ on $E$ which is associated to $D$.
\end{definition}
\begin{remark}
The reason we are interested in the flat vector bundle $(\det E,\det \nabla)$ comes from the natural identifications:
\begin{eqnarray*}
\det C(M^K, E) &\cong& \otimes_{\Delta \in K} \det \Pi^{\dim \Delta} E_{v_0(\Delta)}\\
 &\cong& \otimes_{\Delta \in K} \Gamma_{\textrm{flat}}(\Delta,\det E)^{(-1)^{\dim \Delta}},
\end{eqnarray*}
where $\Gamma_{\textrm{flat}}(\Delta,\det E)$ denotes the vector space of locally constant
sections with respect to the flat connection $\det \nabla$.
\end{remark}

\begin{definition}
Let $E$ be a $\mathbb{Z}_2$-graded vector bundle over $M$, equipped with a flat superconnection and
$\mu$ a flat, non-vanishing section of $(\det (E\oplus E^*), \det (\nabla\oplus \nabla^*))$.

The norm $\tau^{\mu}$ on $\det \C_K(M,E)$ associated to $\mu$ is given via the identification
\begin{eqnarray*}
 \det \C_K(M,E) &\cong & \det C(M^K,E) \otimes \det C(M^K,E^*)\\
 & \cong & \otimes_{\Delta \in K} \Gamma_{\mathrm{flat}}(\Delta, \det (E\oplus E^*))^{(-1)^{\dim \Delta}}.
\end{eqnarray*}
More precisely, we define a norm on each $\Gamma_{\mathrm{flat}}(\Delta, \det (E\oplus E^*))$  by requiring that $\mu|_\Delta$ has norm one.
This induces a norm on $\otimes_{\Delta \in K} \Gamma_{\mathrm{flat}}(\Delta, \det (E\oplus E^*))^{(-1)^{\dim \Delta}}$ and $\tau^\mu$ is the corresponding norm on $ \det \C_K(M,E) $ under the isomorphism above.
\end{definition}

\begin{lemma}
The construction of the norm $\tau^\mu$ satisfies the following properties.
\begin{enumerate}
 \item The norm $\tau^{\lambda \mu}$ on $\det \C_K(M,E)$ associated to a multiple of a flat, non-vanishing section $\mu$
of $(\det (E\oplus E^*), \det (\nabla\oplus \nabla^*))$ relates to $\tau^{\mu}$ by
\begin{align*}
 \tau^{\lambda \mu} = |\lambda|^{-\chi(M)} \tau^{\mu},
\end{align*}
where $\chi(M)$ denotes the Euler characteristic of $M$.
 \item The vector bundle $\det(E\oplus E^*)$ is canonically isomorphic to the trivial line bundle and moreover,
 the section corresponding to the constant function one is flat with respect to the connection $\det(\nabla \oplus \nabla^*)$ 

\item If $M$ is odd dimensional then
the norms on $\det \C_K(M,E)$ associated to any two flat, non-vanishing sections
of $(\det(E\oplus E^*),\det(\nabla \oplus \nabla^*))$ coincide.
\end{enumerate}
\end{lemma}

\begin{proof}
In order to prove the first claim we will first show that over each simplex $\Delta \in K$ the norms $\tau_\Delta^{\lambda \mu}$ and
$\tau_{\Delta}^\mu$ on  $\Gamma_{\mathrm{flat}}(\Delta, \det (E\oplus E^*))$ are related by
\[ \tau_\Delta^{\lambda \mu}=|\lambda|^{-1} \tau_\Delta^{ \mu}.\]
Indeed we see that:
\[\tau_\Delta^{ \mu}(\mu)=1= \tau_\Delta^{\lambda \mu}(\lambda \mu)=|\lambda| \tau_\Delta^{\lambda \mu}(\mu).\]
It follows that the corresponding norms on $\Gamma_{\mathrm{flat}}(\Delta, \det (E\oplus E^*))^{(-1)^{\dim \Delta}}$
are related by:
\[ \tau_\Delta^{\lambda \mu}=|\lambda|^{-(-1)^{\dim \Delta}} \tau_\Delta^{ \mu}.\]
By tensoring over all simplices we obtain that:
\[ \tau^{\lambda \mu}=|\lambda|^{-\sum_{\Delta \in K}(-1)^{\dim \Delta}} \tau^{ \mu}=|\lambda|^{-\chi(M)} \tau^{ \mu}.\]

The second claim follows from the fact that there is a canonical isomorphism
\[\det E^*\cong (\det E)^*\]
under which the connection $ \det (\nabla^*)$ corresponds to $\det(\nabla)^*.$
The last claim is a direct consequence of the first part and the fact that odd dimensional manifolds have zero Euler characteristic.
\end{proof}

\begin{definition}\label{definition:torsion}
Let $M$ be a closed manifold of odd dimension and $(K,\phi)$
a smooth triangulation of $M$ together with a total ordering of its vertex set.
Moreover, let $E$ be a $\mathbb{Z}_2$-graded vector bundle equipped with a flat superconnection $D$.
The Reidemeister torsion $\tau_R$ is the norm on the determinant line $\det H(M,E)$
obtained via the identification
\begin{align*}
T_K: \det \C_K(M,E) \to (\det H(M,E))^{\otimes 2}.
\end{align*}
More explicitly, we set:
\begin{align*}
\tau_R(x):= \sqrt{\tau^{\mu}(T_K^{-1}(x\otimes x))},
\end{align*}
where $\mu$ is an arbitrary flat non-vanishing section of  the flat bundle $(\det(E\oplus E^*),\det(\nabla \oplus \nabla^*))$.
\end{definition}

\subsection{Independence of the choices}

We will prove here the main result of the paper:

\begin{theorem}\label{theorem:invariance}
Let $M$ be a closed, orientable, odd dimensional manifold and $E$ a $\mathbb{Z}_2$-graded vector bundle over $M$, equipped with a flat superconnection.
The Reidemeister norm $\tau_R$ on $\det H(M,E)$ is independent of the choice of a smooth triangulation.
\end{theorem}
\begin{proof}
Choose a smooth triangulation $(K,\phi)$ of $M$ equipped with a total ordering of its vertex set.
The idea of the proof is to use the spectral sequence $\mathcal{E}^{p,\overline{	q}}_{r}$ associated to the filtration
\begin{align*}
F_p\mathcal{C}_K(M,E) := \prod_{k \ge p} \left(C^{k}(M^K,E)\oplus C^k(M^K,E^*)\right),
\end{align*}
in order to reduce the statement to the classical invariance statement for ordinary Reidemeister torsion.

By Proposition \ref{proposition:spectral_sequence} we know that the composition of isomorphisms

\[\det \mathcal{C}_K(M,E) \cong \det H(\mathcal{C}_K(M,E)) \cong \det H(M,E) \otimes \det H(M,E^*),\]
is equal to the composition
\[\det \mathcal{C}_K(M,E) \cong \det {\cal E}_1 \cong \det {\cal E}_2 \cong \cdots 
 \cdots \cong \det {\cal E}_{\infty} \cong 
 \det H(M,E) \otimes \det H(M,E^*).\]

On the other hand, Lemma \ref{lemma:fiber_cohomology} implies that the norm on $\det {\cal E}_1$ induced from that on $\det \mathcal{C}_K(M,E)$
is equal to the norm obtained by the isomorphism ${\cal E}_1 \cong \mathcal{C}_K(M,H_{\partial}(E))$, where  $H_{\partial}(E)$ is
the cohomology bundle of $E$ with the induced flat connection.
In particular, the norm on $\det {\cal E}_2$ will coincide with the norm obtained via the isomorphism
\begin{align*}
\mathcal{E}_2\cong  H(\mathcal{C}_K(M,H_{\partial}(E))\cong H(M^K, H_{\partial}(E)) \oplus H(M^K,  H_{\partial}(E)^*).
\end{align*}
By the triangulation independence of the usual Reidemeister torsion it is known that the norm on 
\[ \det \left(H(M^K,  H_{\partial}(E)) \oplus H(M^K,  H_{\partial}(E)^*)\right)\]
does not depend on the chosen triangulation in the sense that there is a unique norm on
\begin{align*}
\det H(M,H_{\partial}(E))\otimes \det H(M,H_{\partial^*}(E^*))
\end{align*}
such that for any smooth triangulation $(K,\phi)$, equipped with a total ordering of its vertex set, the isomorphism
\begin{align*}
\det H(M,H_{\partial}(E))\otimes \det H(M,H_{\partial^*}(E^*)) \cong  \det \left(H(M^K,  H_{\partial}(E)) \oplus H(M^K,  H_{\partial}(E)^*)\right)
\end{align*}
is an isomorphism of normed vector spaces, see \cite{Faber}.

This implies that the norm on
\begin{align*}
\det \mathcal{C}_K(M,E) \cong \det \mathcal{E}_\infty \cong \det H(M,E) \otimes \det H(M,E^*) 
\end{align*}
also does not depend on the chosen triangulation $(K,\phi)$ and the ordering of its vertex set.
Consequently the norm $\tau_R$ is a well defined invariant of the flat superconnection $D$.
\end{proof}

\begin{remark}
There is a natural decreasing filtration on the complex $\Omega(M,E)$, given by
\[F_p \Omega(M,E):= \oplus_{k\geq p}  \Omega^k(M,E).\]
This filtration induces a spectral sequence ${\cal E}^{p,\overline{q}}_r$ with a natural isomorphism:
\[ {\cal E}^{p,\overline{q}}_2 \cong H^p(M,H^{\overline{q}}_\partial(E)),  \]
where
$H^{\overline{q}}_\partial(E)$ denotes the cohomology vector bundle with the induced flat connection.
This computation induces a canonical isomorphism
\begin{align*}
 \det  H(M,H_{\partial}(E)) \cong \det \mathcal{E}_2 \cong \det \mathcal{E}_\infty \cong \det H(M,E).
\end{align*}

Since each $H^{\overline{q}}_\partial(E)$ is a flat vector bundle in its own right, the usual Reidemeister torsion yields a norm on
$\det  H(M,H_\partial(E))$.
It turns out that the isomorphism $\det H(M,H_{\partial}(E)) \cong \det H(M,E)$ is an isomorphism of metric vector spaces, 
as one can show using the following result:
\end{remark}

\begin{proposition}\label{PD}
The following diagram commutes:
\begin{align*}
\xymatrix{
\det H(M,H_{\partial^*}(E^*)) \ar[r]^{\cong} \ar[d]^{\cong} & \det H(M,H_{\partial}(E)) \ar[d]^{\cong}\\
\det H(M,E^*) \ar[r]^{\cong}& \det H(M,E)
}
\end{align*}
Here the horizontal isomorphisms are induced by Poincar\'e duality and the vertical maps are the canonical isomorphisms coming from the 
spectral sequence.
\end{proposition}

\begin{proof}
This is a direct application of Lemma \ref{lemmacompatibility}. The only hypothesis which is not obviously satisfied is the 
compatibility of the filtrations in cohomology. Thus, we only need to prove that if we set $n=\dim M$, the filtration on $H(M,E^*)$ is given by:

\[F_p(H(M,E^*))=\{ [b] \in H(M,E^*): \langle a,b\rangle =0 \text{ if } [a] \in F_{n-p+1}(H(M,E))        \}.\]
Clearly, the left hand side of the equation is contained in the right hand side. 
In order to prove the other inclusion we will use Hodge decomposition.
We choose a Riemannian metric on $M$ as well as a fiber metric on $E$. 
We claim that for an element $[b]$ in the right hand side of the equation above, the harmonic representative $b$ of the 
cohomology class belongs to $F_p(\Omega(M,E^*))$. 
We will argue by contradiction. Suppose the opposite is true and consider the smallest $p$ for which this happens. We may assume that 
the cohomology class $[b]$ is homogeneous (with respect to the total $\mathbb{Z}_2$-grading) and write
\[b=b_{p-1}+ b_p+\dots ,\]
where $b_k$ is a differential form of degree $k$. Consider $\ast b_{p-1}\in \Omega^{n-p+1}(M,E)$ and use Hodge decomposition to write
\[\ast b_{p-1}=x+y+z,\]
with $x \in {\cal H}(M,E)$, $y \in  \im D$ and  $z \in \im D^*$, where ${\cal H}(M,E)$ denotes the space of harmonic forms. We know that
\[\langle b, \ast b_{p-1}\rangle =\langle b_{p-1} ,\ast b_{p-1} \rangle \neq 0. \]
On the other hand, since $b$ is a harmonic form we have
\[ \langle b,y\rangle =\langle b,z \rangle =0,\]
so we conclude that
\[0\neq \langle b, x\rangle= \langle [b], [x]\rangle, \]
but this contradicts the hypothesis because $[x]\in F_{n-p+1}H(M,E)$.

\end{proof}

From Proposition \ref{PD} and the proof of Theorem \ref{theorem:invariance} we conclude:
\begin{corollary}\label{equal}
The Reidemeister torsion of $E$ coincides with the Reidemeister torsion of the cohomology flat vector bundles.
More precisely, the isomorphism:
\[ \det  H(M,H^{\overline{0}}_\partial(E))\otimes \det  H(M,H^{\overline{1}}_\partial(E)))\cong \det H(M,E),\]
is an isomorphism of metric vector spaces.
\end{corollary}

\subsection{Invariance under quasi-isomorphism}

As we mentioned before, the set of flat superconnections over $M$ forms a $dg$-category which we denote
$\RRep(TM)$. 
Every morphism $\phi:E \rightarrow E'$ decomposes as a  sum:
\begin{align*}
\phi = \phi_0 + \phi_1 + \cdots,
\end{align*}
where $\phi_i \in \Omega^{i}(M,\End(E,F))$. 
We denote the subset of cocycles of even degree in $\RHom(E,E')$ by $\Hom(E,E')$.
Observe that for $\phi \in \Hom(E,E')$ the component $\phi_0$ is a vector bundle map from $E$ to $E$' which is compatible with the fiberwise
differentials.

\begin{definition}
An element $\phi \in \Hom(E,E')$ is a quasi-isomorphism if its component
\begin{align*}
 \phi_0: E \to E'
\end{align*}
induces an isomorphism between the fiberwise cohomologies.
\end{definition}

\begin{remark}
By a standard spectral sequence argument one can check that a quasi-isomorphism $\phi \in \Hom(E,E')$ induces and isomorphism in cohomology
$[\phi]: H(M,E) \to H(M,E')$.
\end{remark}

\begin{proposition}\label{quasi-iso}
Let $M$ be a closed, orientable, odd dimensional manifold and $\phi$ a quasi-isomorphism between flat superconnections $E$ and $E'$.
Then, the induced isomorphism
\begin{align*}
 \det [\phi]: \det H(M,E) \to \det H(M,E')
\end{align*}
is compatible with the Reidemeister torsions on $\det H(M,E)$ and $\det H(M,E')$, respectively.
\end{proposition}

\begin{proof}
The $\A_\infty$-functor constructed in  \cite{BS,AS2} gives a chain map:
\begin{align*}
 \int[\phi]: \hat{C}(M,E) \to \hat{C}(M,E').
\end{align*}

Since the pullback operation is functorial, we can choose a triangulation $(K,\phi)$ of $M$, and restrict the morphism $\int [\phi] $ to obtain a chain map: 
\begin{align*}
 \int[\phi]: C(M^K, E) \to C(M^K,E'),
\end{align*}
making the diagram
\begin{align*}
 \xymatrix{
C(M^K, E) \ar[r]^{\int[\phi]} & C(M^K, E') \\
\Omega(M,E) \ar[r]^{\phi} \ar[u]& \Omega(M,E') \ar[u]
}
\end{align*}
commute up to a homotopy $h: \Omega(M,E) \to C(M^K,E')$. Similarly, we have a chain map
\begin{align*}
 \int[\phi^*]: C(M^K,(E')^*) \to C(M^K,E^*),
\end{align*}
making the diagram
\begin{align*}
 \xymatrix{
C(M^K,(E')^*) \ar[r]^{\int[\phi^*]} & C(M^K,E^*) \\
\Omega(M,(E')^*) \ar[r]^{\phi^*} \ar[u]& \Omega(M,E^*) \ar[u]
}
\end{align*}
commute up to a homotopy $\tilde{h}: \Omega(M,(E')^*) \to C(M^K,E^*)$.

The chain maps $\int[\phi]$ and $\int([\phi^*])$ respect the filtrations given by cochain degrees and induce isomorphisms on the first 
sheets of the corresponding spectral sequences, i.e. 
\begin{eqnarray*}
C(M^K, H_{\partial}(E)) = \mathcal{E}_1C(M^K,E)  &\cong& \mathcal{E}_1C(M^K, E') = C(M^K, H_{\partial}(E')),\\
C(M^K,(E')^*)) = \mathcal{E}_1C(M^K,(E')^*)  &\cong& \mathcal{E}_1C(M^K,E^*) = C(M^K, H_{\partial^*}(E^*)).
\end{eqnarray*}
Together, these yield an isomorphism
\begin{align*}
\det \mathcal{E}_1\mathcal{C}_K(M,E) \cong \det \mathcal{E}_1\mathcal{C}_K(M,E').
\end{align*}
We saw in the proof of Theorem \ref{theorem:invariance} that the norm on $\det \mathcal{E}_1\mathcal{C}_K(M,E)$
coincides with the norm obtained via the natural identification
\begin{align*}
 \det \mathcal{E}_1 \mathcal{C}_K(M,E) \cong \det \mathcal{C}_K(M,H_{\partial}(E)),
\end{align*}
where the norm on the latter line comes from some non-vanishing flat section of the bundle $\det(H_{\partial}(E)\oplus H_{\partial^*}(E^*))$ with
flat connection $\det([\nabla]\oplus [\nabla^*])$.
Since $\phi_0$ induces an isomorphism of vector bundles with flat connections
\begin{align*}
 \det(H_{\partial}(E)\oplus H_{\partial^*}(E^*)) \cong \det(H_{\partial}(E')\oplus H_{\partial^*}((E')^*)),
\end{align*}
the isomorphism
\begin{align*}
\det \mathcal{E}_1\mathcal{C}_K(M,E) \to \det \mathcal{E}_1\mathcal{C}_K(M,E')
\end{align*}
obtained above is compatible with the norms. 
Hence so is
\begin{align*}
\det [\phi] \otimes \det [\phi^*]^{-1}: \det H(M,E) \otimes \det H(M,E^*) \to \det H(M,E') \otimes \det H(M,(E')^*). 
\end{align*}
Together with the commutativity of
\begin{align*}
\xymatrix{
\det H(M,(E')^*) \ar[r]^{[\phi^*]} \ar[d] & \det H(M,E^*) \ar[d]\\
\det H(M,E') & \ar[l]_{[\phi]} \det H(M,E),
}
\end{align*}
where the vertical arrows are given by Poincar\'e duality, this implies that
\begin{align*}
\det [\phi]: \det H(M,E) \to \det H(M,E')
\end{align*}
is compatible with the norms given by the Reidemeister torsion.
\end{proof}

\begin{lemma}
Let $M$ be a closed, orientable manifold of odd dimensions and $E$ and $E'$ two $\mathbb{Z}_2$-graded vector bundles equipped with flat superconnections.

\begin{enumerate}
\item The isomorphism
\begin{align*}
\det {\cal D}:\det H(M,E) \to \det H(M,E^*)
\end{align*}
induced by Poincar\'e duality maps the Reidemeister torsion of $E$ to the Reidemeister torsion of $E^*$.

\item The natural isomorphism
\begin{align*}
 \det H(M,E\oplus F) \to \det H(M,E) \otimes \det H(M,F)
\end{align*}
maps the Reidemeister torsion of $E\oplus F$ to the product of the Reidemeister torsion for $E$ and $F$, respectively.
\end{enumerate}
\end{lemma}

\begin{proof}
In order to prove the first statement we choose a triangulation $(K,\phi)$ together with a total ordering of the vertex set and observe that:

\[\C_K(M,E) = C(M^K,E)\oplus C(M^K,E^*)= \C_K(M,E^*)\]

This shows that the metric induced on $\det H(M,E) \otimes \det H(M,E^*)$ by the Reidemeister torsion of $E$ coincides with the 
metric induced on $\det H(M,E) \otimes \det H(M,E^*)$ by the Reidemeister torsion of $E^*$. 

Let us denote by $\det \cal{D}$ the isomorphism
\[ \det {\cal D}: \det H(M,E) \rightarrow \det H(M,E^*)\]
induced by Poincar\'e duality. Consider the commutative diagram

\begin{align*}
\xymatrix{
\det H(M,(E)) \otimes \det H(M,E) \ar[rrd]^{\,\,\,\,\,\,\,\,\,\,\,\,\,\,\,\, \id \otimes \det {\cal D}} \ar[dd]_{ \det {\cal D}\otimes  \det {\cal D}} & & \\
&& \det H(M,E)\otimes \det H(M,E^*) \ar[dll]^{\det {\cal D}\otimes \id }\\
\det H(M,E^*)\otimes H(M,E^*) && ,
}
\end{align*}
where we give $ \det H(M,E) \otimes \det H(M,E)$  and $\det H(M,E^*) \otimes \det H(M,E^*)$ the norms corresponding to the Reidemeister
torsions of $E$ and $E^*$, respectively. Then, by construction, the maps $\id \otimes \det {\cal D}$ and $\det {\cal D} \otimes \id$ are maps of
normed vector spaces. Since the diagram is commutative we conclude that so is $\det {\cal D} \otimes \det{\cal D}$. This clearly implies that:
\[\det {\cal D}: \det H(M,E)\rightarrow \det H(M,E^*)\]
preserves the norm. The second claim is immediate from the construction.
\end{proof}

\appendix

\section{Homological algebra of determinant lines}\label{appendix1}

Here we collect some well-known properties of determinant lines of graded vector spaces.
In the following all the $\mathbb{Z}_2$-graded vector spaces under consideration are assumed to be finite dimensional.
If $V$ is a $\mathbb{Z}_2$-graded vector space, $\Pi V$ denotes the $\mathbb{Z}_2$-graded vector space
given by $V$ with parity reversed, i.e. $(\Pi V)^{\overline{0}} = V^{\overline{1}}$ and $(\Pi V)^{\overline{1}} = V^{\overline{0}}$.

\begin{definition}
The determinant line $\det V$ of a $\mathbb{Z}_2$-graded vector space $V$ is the one dimensional vector space
\begin{align*}
\det V := \wedge^{\mathrm{top}} V^{\overline{0}} \otimes \wedge^{\mathrm{top}} (V^{\overline{1}})^{*}.
\end{align*}
\end{definition}

\begin{lemma}
The lines  $\det(V)$ and $\det(\Pi V^*)$ are the same.
\end{lemma}

\begin{lemma}
Let 
\begin{align*}
 \xymatrix{
0\ar[r] & U \ar[r] & V \ar[r] & W \ar[r] & 0
}
\end{align*}

be a short exact sequence of $\mathbb{Z}_2$-graded vector spaces.

Then there is a natural isomorphism
\begin{align*}
 \det(V) \cong \det(U)\otimes \det(W).
\end{align*}
\end{lemma}

\begin{lemma}
Let $(C,d)$ be a $\mathbb{Z}_2$-graded complex and $H(C)$ its cohomology. There is a natural isomorphism
\begin{align*}
\det(C) \cong \det(H(C)).
\end{align*}
\end{lemma}

\begin{remark}
This Lemma follows from the previous one by considering the following two short exact sequences:
\begin{align*}
 \xymatrix{
0 \ar[r] & B \ar[r] & Z \ar[r] & H \ar[r] & 0 \\
0 \ar[r]  & Z \ar[r] & C \ar[r]^{d} & \Pi B \ar[r] & 0,
}
\end{align*}
where $Z$ and $B$ denote the $\mathbb{Z}_2$-graded vector spaces given by the cocycles and coboundaries of $(C,d)$, respectively.
\end{remark}

\begin{lemma}
Let $V$ be a $\mathbb{Z}_2$-graded vector space with a bounded filtration
\begin{align*}
V \supset \cdots \supset F_pV \supset F_{p+1}V \supset \cdots \supset 0.
\end{align*}
Then there is a natural isomorphism
\begin{align*}
\det(V) \cong \det(GV),
\end{align*}
between the determinant line of $V$ and the determinant line of the associated graded $GV$ of $V$.
\end{lemma}

Let $(C,d)$ be a $\mathbb{Z}_2$-graded complex with a compatible filtration $(F_pC)$
\begin{align*}
C \supset \cdots \supset F_pC \supset F_{p+1}C \supset \cdots \supset 0,
\end{align*}
which is bounded. We denote the $k$'th sheet of the associated spectral sequence by $\mathcal{E}_kC$.
Observe that $\mathcal{E}_1C$ is the cohomology $H(GC)$ of the associated graded $GC$ of $C$. Moreover, since the filtration
is bounded, $\mathcal{E}_{\infty}C$ is isomorphic to the associated graded $G H(C)$ of the cohomology
$H(C)$ of $C$ -- the latter equipped with the filtration inherited from $C$.

\begin{proposition}\label{proposition:spectral_sequence}
Let $(C,d)$ be a finite dimensional $\mathbb{Z}_2$-graded complex with a bounded filtration $(F_pC)$.
Then the diagram of natural isomorphisms
\begin{align*}
\xymatrix{
\det(C) \ar[rrr] \ar[d] &&& \det(H(C)) \ar[d]\\
\det(GC) \ar[d] &&& \det(G H(C)) \ar[d]\\
\det(H(GC))=\det(\mathcal{E}_1C) \ar[r] & \det(\mathcal{E}_2 C) \ar[r] & \cdots \ar[r] & \det \mathcal{E}_{\infty}C
}
\end{align*}
commutes up to sign.
\end{proposition}

\begin{remark}
Proofs of this statement can be found in \cite{Maumary} or \cite{Freed}.
\end{remark}

\section{Duality of spectral sequences}

Here we will prove some lemmas that are be useful in showing the compatibility of the Poincar\'e duality
isomorphism with the canonical isomorphism between the determinant of a (finite dimensional) complex
and its cohomology. 

\begin{lemma}\label{easypairing}
Let $A$ and $B$ be finite dimensional $\mathbb{Z}_2$-graded complexes and assume that there is a parity $\overline{n}$ pairing
\[\langle -,- \rangle :A \otimes B \rightarrow \mathbb{R},\]
which is nondegenerate and such that
\[\langle \partial a , b \rangle = -(-1)^{|a|} \langle a , \partial b \rangle \]
holds.

Then, the pairing induced in cohomology is nondegenerate and, if the parity of the pairing is odd, the diagram
\begin{align*}
\xymatrix{
\det A \ar[r]^{\cong} \ar[d]^{\cong} & \det B \ar[d]^{\cong}\\
\det H(A) \ar[r]^{\cong}& \det H(B)
}
\end{align*}
commutes. Here the horizontal isomorphisms are induced by the pairing and the vertical ones are the canonical isomorphisms.
\end{lemma}

\begin{proof}
First let us prove that the pairing in cohomology is nondegenerate. For this we choose a metric on $A$ which gives a Hodge decomposition
\[A={\cal{H}}(A)\oplus \partial (A) \oplus \partial^* (A),\]
where ${\cal H}(A)$ denotes the space of harmonic elements.
The corresponding metric on $B$ gives a decomposition
\[B={\cal{H}}(B)\oplus \partial (B) \oplus \partial^* (B).\]
Since the metric on $B$ is the one induced by that on $A$ we know that:
\[\langle {\cal{H}}(A) , \partial(B)\oplus \partial^*(B)\rangle =0,\]
and
\[\langle {\cal{H}}(B) , \partial(A)\oplus \partial^*(A)\rangle =0.\]
But since the pairing is nondegenerate we conclude that the pairing between ${\cal H}(A)$ and
${\cal H}(B)$ is nondegenerate. This implies that the pairing in cohomology is also nondegenerate.

In order to prove that the diagram above commutes, it suffices to observe that the diagram

\begin{align*}
\xymatrix{
\det A \ar[r]^{\cong} \ar[d]_{\cong} & \det \Pi B^* \ar[d]^{\cong} \\
\det H(A) \ar[r]^{\cong}& \det H(\Pi B^*).
}
\end{align*}
commutes. This is true because the isomorphisms between $\det A$ and $\det H(A)$ (respectively for $B$ and $H(B)$) are canonical. 
The rest follows from the observation
that the diagram
\begin{align*}
\xymatrix{
\det B^* \ar[r]^{=} \ar[d]_{\cong}& (\det B)^* \ar[d]^{\cong}\\
\det H(B^*) \ar[r]^{\cong} & (\det H(B))^*
}
\end{align*}
commutes as well.
\end{proof}

\begin{lemma}\label{lemmacompatibility}
Let $A = (A^{k,\overline{p}})$ and $B=(B^{k,\overline{p}})$
be $\mathbb{Z}_2$-graded complexes with an additional $\mathbb{Z}$-grading bounded between $0$ and $n\ge 0$.
Assume further that $A$ and $B$ have finite dimensional cohomology
and that the respective differentials preserve the filtrations $F_p(A)$ and $ \tilde{F}_p(B)$ given by
\[F_p(A)=\bigoplus_{k\geq p}A^{k,\bullet}, \quad \tilde{F}_p(B)=\bigoplus_{k\geq p}B^{k,\bullet}.\]

Let
\[\langle \, , \, \rangle: A\otimes B \rightarrow \mathbb{R} ,\]
be a pairing of bidegree $(-n,\overline{0})$
such that
\begin{enumerate}
\item The differentials are skew-adjoint with respect to the pairing, namely: \[\langle \partial a, b \rangle =-(-1)^{|a|}\langle a, \partial b \rangle.\]
\item The induced pairing in cohomology 
\[\langle - , - \rangle:H (A)\otimes H(B)
\rightarrow \mathbb{R} \]
is nondegenerate. 
\end{enumerate}
Consider the spectral sequences $\mathcal{E}_r$ and $\tilde{\mathcal{E}}_r$ associated to the filtrations $F_p(A)$ and $\tilde{F}_p(B)$. Then:
\begin{enumerate}
\item For $r\geq 0$ there is a pairing of bidegree $(-n,\overline{0})$:
\[ \langle \, ,\, \rangle :\mathcal{E}_r \otimes \tilde{\mathcal{E}}_r \rightarrow \mathbb{R},\]
given by the formula
\[ \langle [a],[b]\rangle :=\begin{cases}
\langle a,b\rangle \text{ if } [a] \in E^{p,\overline{q}}_r \text{ and } [b]\in E^{n-p,\overline{q}}_r,\\
0 \text{ otherwise. }
\end{cases}\]

\item For each $r$ the pairing makes the differentials $\partial_r$ on the $r$-th sheets skew-adjoint, namely:
 \[\langle \partial_r [a], [b] \rangle = -(-1)^{|a|}\langle [a], \partial_r [b] \rangle.\]
\item The pairing between the $(r+1)$-th pages is the one induced in cohomology by the paring between the page $r$-th pages.
\item If there exists  $m\geq 0$ such that the $m$-th pages  are finite dimensional and the paring between them is nondegenerate then the same is true for $r \geq m$.
\item Assume that $n$ is odd, the $r$-th pages are finite dimensional and the pairing between them is nondegenerate.
If the filtrations induced in cohomology are compatible in the sense that
\[\tilde{F}_p(H(B))=\{[b] \in H(B): \langle [a],[b]\rangle=0 \text{ for } [a]\in F_{n-p+1}(H(A))\},\]
then the following diagram commutes
\begin{align*}
\xymatrix{
\det \mathcal{E}_r \ar[r]^{\cong} \ar[d]_{\cong} & \det \tilde{\mathcal{E}}_r \ar[d]^{\cong}\\
\det H(A) \ar[r]^{\cong}& \det H(B).
}
\end{align*}
Here the vertical isomorphisms are the canonical isomorphisms, while the horizontal ones are induced by the pairing.
\end{enumerate}
\end{lemma}

\begin{proof}
For the first claim we need to show that the pairing from (1.) is well defined. We know that
\[\mathcal{E}^{p,q}_r:= \frac{\{a \in F_p(A^{\overline{p+q}}): \partial(a) \in F_{p+r}(A^{\overline{p+q+1}})\}}{F_{p+1}(A^{\overline{p+q}})+\partial (F_{p-r+1}(A^{\overline{p+q-1}})) }\]
and 
\[\tilde{\mathcal{E}}^{p',q'}_r:= \frac{\{b \in \tilde{F}_{p'}(B^{\overline{p'+q'}}): \partial(b) \in \tilde{F}_{p'+r}(B^{\overline{p'+q'+1}})\}}{\tilde{F}_{p'+1}(B^{\overline{p'+q'}})+\partial (\tilde{F}_{p'-r+1}(B^{\overline{p'+q'-1}}))}.  \]

We can assume that $p'=n-q$ and $\overline{q'}=\overline{q}$ since otherwise the pairing is zero.
In order to prove that the pairing is well defined we need to show that
\[ \langle F_{p+1}(A^{\overline{p+q}})+\partial (F_{p-r+1}(A^{\overline{p+q-1}})),\{b \in \tilde{F}_{p'}(B^{\overline{p'+q'}}): \partial(b) \in \tilde{F}_{p'+r}(B^{\overline{p'+q'+1}})\} \rangle =0 , \]
and that:
\[ \langle \tilde{F}_{p'+1}(B^{\overline{p'+q'}})+\partial (\tilde{F}_{p'-r+1}(B^{\overline{p'+q'-1}})), \{a \in F_p(A^{\overline{p+q}}): \partial(a) \in F_{p+r}(A^{\overline{p+q+1}})\}\rangle =0.\]
By symmetry, it is enough to prove the first equation.
Consider $a \in  F_{p+1}(A^{\overline{p+q}})$ and $b \in \{b \in \tilde{F}_{p'}(B^{\overline{p'+q'}}): \partial(b) \in \tilde{F}_{p'+r}(B^{\overline{p'+q'+1}})\} $. Then
\[\langle a,b \rangle =0\]
because $p+1+p'> n$ and the pairing has bidegree $(-n,\overline{0})$. Next, let us assume that $c \in \partial (F_{p-r+1}(A^{\overline{p+q-1}}))$ 
and $b \in \{b \in \tilde{F}_{p'}(B^{\overline{p'+q'}}): \partial(b) \in \tilde{F}_{p'+r}(B^{\overline{p'+q'+1}})\} $. Hence
\[\langle \partial c , b \rangle=\pm \langle c, \partial b \rangle=0, \]
because $c \in F_{p-r+1}(A^{\overline{p+q}})$, $\partial b \in F_{p'+r}(B^{\overline{p'+q'}})$ and the pairing has bidegree $(-n, \overline{0})$. We conclude that the pairing is
well defined.

The second claim is an immediate consequence of the fact that the pairing between $A$ and $B$ makes the differentials skew-adjoint.
The third claim follows from the explicit formula for the pairing.
The fourth claim is a direct application of Lemma \ref{easypairing}
In order to prove the last claim we will show that the following diagram commutes:
\begin{align*}
\xymatrix{
\det \mathcal{E}_r \ar[r]^{\cong} \ar[d]_{\cong} & \det \tilde{\mathcal{E}}_r \ar[d]^{\cong}\\
\det \mathcal{E}_{r+1} \ar[r]^{\cong} \ar[d]_{\cong} & \det \tilde{\mathcal{E}}_{r+1} \ar[d]^{\cong}\\
\vdots \ar[d]_{\cong} & \vdots \ar[d]^{\cong}\\
\det \mathcal{E}_\infty \ar[r]^{\cong} \ar[d]_{\cong} & \det \tilde{\mathcal{E}}_\infty \ar[d]^{\cong}\\
\det G H(A) \ar[r]^{\cong} \ar[d]_{\cong} & \det G H(B) \ar[d]^{\cong}\\
\det H(A) \ar[r]^{\cong}& \det H(B)
}
\end{align*}

The fact that all but the last diagram commute is a direct application of Lemma \ref{easypairing}. It remains to show that the last diagram commutes. 
From the explicit formula for the pairing on the spectral sequence we know that the pairing at the level of associated graded vector spaces in the 
cohomology induced by the pairing on the spectral sequences is given by
\[ \langle [a],[b]\rangle :=\begin{cases}
\langle a,b\rangle \text{ if } [a] \in F_p(H(A))/F_{p+1}(H(A)) \text{ and } [b]\in \tilde{F}_{n-p}(H(B))/\tilde{F}_{n-p+1}(H(B)) ,\\
0 \text{ otherwise. }
\end{cases}\]

That the last diagram commutes can be seen as follows: first, observe that the pairing yields an isomorphisms
\begin{align*}
 H(A) \cong \Pi H(B)^*
\end{align*}
which maps $F_pH(A)$ to  $\tilde{F}_p H(B)^*:=(\tilde{F}_{n-p+1}H(B))^{\circ}$, i.e. the annihilator of $\tilde{F}_{n-p+1}H(B)$ in $H(B)^*$.
The isomorphism $\phi: \det H(A) \cong \det H(B)$ in the diagram above is induced from this isomorphism.
Moreover, we obtain the following commutative diagram
\begin{align*}
\xymatrix{
\det H(A) \ar[r]^{\det [\phi]} \ar[d] & \det (\Pi H(B)^*) \ar[d] \ar[r]^{=} & \det H(B) \ar[d]\\
\det G(H(A)) \ar[r]^{\det G[\phi]} & \det G(\Pi H(B)^*) \ar[r]^{\cong} & \det G(H(B)),
}
\end{align*}
where the last arrow on the second line comes from the natural identification
\begin{align*}
 \tilde{F}_p H(B)^* / \tilde{F}_{p+1}H(B)^* \cong \left(\tilde{F}_{n-p}H(B) / \tilde{F}_{n-p+1}H(B) \right)^*.
\end{align*}
The composition of $\det G(H(A)) \cong \det G(\Pi H(B)^*) \cong \det G(H(B))$ coincides with the map in the next to last line in the previous diagram.
Hence
\begin{align*}
\xymatrix{
\det G H(A) \ar[r]^{\cong} \ar[d]_{\cong} & \det G H(B) \ar[d]^{\cong}\\
\det H(A) \ar[r]^{\cong}& \det H(B)
}
\end{align*}
commutes as well.

\end{proof}

\end{document}